\documentclass[10pt,oneside,a4paper]{article}

\usepackage{amssymb}
\usepackage{amsmath}
\usepackage{amsthm}
\usepackage{cite}
\usepackage{hyperref}
\usepackage[ruled]{algorithm2e}

\newtheorem{thm}{Theorem}[section]
\newtheorem{cor}[thm]{Corollary}
\newtheorem{prop}[thm]{Proposition}
\newtheorem{lem}[thm]{Lemma}

\theoremstyle{remark}
\newtheorem{rem}[thm]{Remark}

\theoremstyle{definition}

\numberwithin{equation}{section}

\SetKwRepeat{Do}{do}{while} 

\def\RR{\mathbb{R}}
\def\CC{\mathbb{C}}
\def\ZZ{\mathbb{Z}}
\def\NN{\mathbb{N}}
\def\TT{\mathbb{T}}

\DeclareMathOperator{\WB}{WB}

\DeclareMathOperator{\DFT}{DFT}
\DeclareMathOperator{\IDFT}{IDFT}

\def\IM{\mathrm{Im}}
\def\im{\mathrm{i}}
\DeclareMathOperator{\supp}{\mathrm{supp}}

\def\Id{\mathrm{Id}}

\title{Computer assisted proofs for hyperbolic quasi-periodic invariant tori in dissipative twist maps}
\author{
    Victor Linroth
        \thanks
        {
            Department of Mathematics, Uppsala University, Box 480, 751 06 Uppsala,  (Sweden). {\tt victor.linroth@math.uu.se}.
        }
}

\begin{document}

\maketitle

\begin{abstract}
    This paper outlines an approach for proving existence of hyperbolic quasi-periodic invariant tori using computer assisted methods based on an a posteriori KAM-like theorem. We implement it for the dissipative standard map following the approach and give computer assisted proofs of existence of invariant circles for some parameter choices.
  \end{abstract}

\section{Introduction}

The problem of rigorously proving existence of invariant tori has received a lot of attention in the field of \textit{KAM theory}, named after A.N. Kolmogorov, V.I. Arnold and J.K. Moser for their work showing the persistence under perturbation of certain invariant tori in integrable Hamiltonian systems. A key idea to make this work was to only consider invariant tori where the inner dynamics fulfill a form of non-resonance condition. Since the original work of these authors, their methods have been greatly expanded upon and transferred to other settings. Here we find the so called \textit{a posteriori KAM theory} that removes the need for the underlying systems to be on the form of a small perturbation of an integrable system and replaces this with the existence of a tori that is close to invariant, see e.g. \cite{Llave_99, Llave_Gonzalez_Jorba_05,Calleja_Celletti_Llave_22} or the book \cite{Haro_16}. This has greatly increased the scope of invariant tori that can be examined but the estimates needed for proving existence quickly become too complicated to do by hand and the assistance of computers has been introduced to overcome this complexity. Here forms of validated numerics, such as interval analysis and the like, has seen a lot of use. Some examples can be found in \cite{Figueras_Haro_Luque_16,Figueras_Haro_Luque_20,Celletti_Chierchia_07,Valvo_Locatelli_22,Caracciolo_Locatelli_20,Caracciolo_Locatelli_21,Figueras_Haro_12,Capinski_Simo_12}.

The core of the methods used in KAM theory is the idea that given an approximate solution to a certain type of functional equation, one can try to make it better by adding a correction that is obtained by solving a linearized equation describing the true difference between the approximate and the exact solution. This approach closely resembles the Newton method for finding zeros of differentiable functions and can indeed be seen as a generalization to an infinite dimensional setting. Like the Newton method, the methods of KAM theory can both be used to prove existence of solutions and as a way of obtaining highly accurate approximations when implemented as a numerical scheme. Much work has been done in this area, in particular under scope of parameterization methods.

The original setting of Hamiltonian dynamics gives systems described by symplectic maps but KAM theory has also been expanded to cover other types of maps, like e.g. dissipative maps which is the setting of this paper. Here the persistence of the invariant tori can not be asserted for any sufficiently regular perturbation but one must also include sufficiently many tuning parameters for the system that can be adapted to preserve the tori. For more details on when invariant tori persist see \cite{Figueras_Haro_Luque_20}. The nature of dissipative systems makes it likely that they will contain invariant objects that are attractors and they can be very convenient from a numerical perspective. If an attractor is also a normally hyperbolic manifold it means that points in it's basin of attraction will be mapped closer and closer to the manifold at an exponential rate. For attracting normally hyperbolic tori with transitive inner dynamics this means a single orbit can approximate the tori at very high accuracy for the very low cost of simply iterating the map.

The problem then becomes how to make the rigorous estimates needed using the points on the orbit and a very useful method for doing this was the main topic of a series of papers \cite{Das_16, Das_17, Das_Yorke_18, Saiki_Yorke_18}. The method is referred to as the \textit{weighted Birkhoff method} and consists of calculating the sum from the Birkoff Ergodic theorem with weights sampled from a bump function. It can be shown that regardless of the shape of the bump function, as long as it and the underlying system are smooth, the weighted sum has super-polynomial convergence if the inner dynamics of the invariant tori fulfills the same non-resonance condition needed for KAM theory. This condition also ensures ergodicity meaning the sum converges to the corresponding integral of a chosen function. Different choices of this function gives integrals describing various useful dynamical properties of the tori, e.g. rotation vectors, Fourier coefficients and Lyapunov exponents. In particular we note that this allows us to go from the orbit approximation to an approximation given by Fourier coefficients which opens up the possibility of using the parameterization methods from a posteriori KAM theory for rigorously proving existence of a true invariant torus.

Here it is appropriate to stop and question if it is worth the effort to formulate and implement such a method for computing approximations of attracting tori. After all there are already parameterization methods based on Newton iteration that provide very numerically efficient ways of accurately approximating even more general invariant tori, not just attracting. One answer here is that while these parameterization methods are indeed generally very efficient they are still implicit methods requiring equation solving at each step while the weighted Birkhoff method is a fully explicit method and as such it is reasonable to believe there are situations where it will outperform implicit methods. This will likely depend a lot on the dimension and regularity of the torus. Furthermore it should be noted that Newton iteration requires a good initial guess in order to work and for complicated tori this is not always easy to obtain, whereas the weighted Birkhoff method only requires a single point in the basin of attraction. This is could be of significance for the study of tori deforming into non-chaotic strange attractors, although often it is very viable to begin with parameter values giving a very regular attractor and then extend the approximations through continuation to parameters giving more and more complicated attractors. Lastly it should also be noted that as of writing there appears to be no readily available software packages for the parameterization method so implementation falls upon the user and here the weighted Birkhoff method should be the easier to implement.

The goal of this paper is to give a general indication for how to first obtain an approximation of an attracting tori using the weighted Birkhoff method, and then how to use this approximation together with a result from a posteriori KAM theory for dissipative dynamical systems in order to prove the existence of a true invariant tori. This is illustrated by implementing the method for the specific case of the dissipative standard map and obtaining the necessary estimates to prove existence of an invariant tori for some example of parameters. Quasiperiodic attractors of the dissipative standard map have been studied before, see e.g. \cite{Calleja_Celletti_10,Calleja_Figueras_12}. Proving the existence of quasi-periodic attractors using a posteriori KAM theory for conformally symplectic maps has been explored in \cite{Calleja_Celletti_Llave_20}, but as it is noted in the article this is done in a non-rigorous way and the authors suggest implementing the verification using interval arithmetic for rigorous results.

Section \ref{sec:setting} gives the necessary background for formulating the theorem giving conditions guaranteeing existence of invariant tori. Section \ref{sec:AnEfPF} presents a collection of numerical methods for working with periodic functions such as parameterizations of tori. This includes the weighted Birkhoff method for approximating Fourier coefficients as well as a handful of rigorous methods based of the discrete Fourier transform that are necessary for obtaining validated estimates of the parameterization. Section \ref{sec:validation} takes the specific case of the dissipative standard map and goes through all the steps of setting up an algorithm for proving the existence of an attracting torus for a given set of parameters and an initial point for the approximation. In Section \ref{sec:results} we look at the output of an implementation of the algorithm for some example parameters.

\section{The Setting and the KAM-like Theorem} \label{sec:setting}

This section aims to give a complete statement of the KAM-like theorem from \cite{Canadell_Haro_2017}. For ease of reference we try to use the same notation and all the definitions should be the same. We begin with presenting the functional environment the theorem is formulated in and then give definitions for quasi-periodic invariant tori, normal hyperbolicity and Diophantine vectors which are fundamental to the theorem.

\subsection{Spaces of Analytic Functions and Maps On the Annulus}

We begin by specifying the functional environment we need to formulate the rest of the definitions and theorems. Denote a \textit{complex strip} of $\TT_{\rho}^d$ of width $\rho > 0$ by
\[
\TT_{\rho}^d = \{ \theta \in \TT_{\CC}^d \mid | \IM \theta_{\ell} | < \rho, \: \ell = 1,\hdots,d \},
\]
where $\TT_{\CC}^d = \CC^d / \ZZ^d$ is the complex $d$-dimensional torus.

Let $(A_{\rho})^m$ be the Banach space of continuous functions $f: \bar{\TT}_{\rho}^d \to \CC^m$, holomorphic on $\TT_{\rho}^d$ such that $f(\TT_{\rho}^d) \in \RR^m$, with norm
\[
\|f\|_{\rho} = \sup_{\theta \in \TT_{\rho}^d} |f(\theta)|.
\]
where $|v| = \max(|v_1|,\hdots,|v_m|)$, $v\in \CC^m$.

Let $(A_{\rho}^r)^m$ be the Banach space of continuous functions $f: \bar{\TT}_{\rho}^d \to \CC^m$, holomorphic on $\TT_{\rho}^d$ such that $f(\TT_{\rho}^d) \in \RR^m$ and whose partial derivatives can be extended continuously to the boundary of $\TT_{\rho}^d$, with norm
\[
\|f\|_{\rho,C^r} = \max_{\ell=0,\hdots,r} \sup_{\theta \in \TT_{\rho}^d} |D^{\ell} f(\theta)|.
\]
The norm $|\cdot |$ applied to multilinear mappings is the norm induced by the supremum norm on $\CC^m$.

The average of a continuous function $f: \TT^d \to \CC^m$ will be denoted as $\langle f \rangle$, i.e.
\[
\langle f \rangle = \int_{\TT^d} f(\theta) d\theta.
\]

Consider an annulus $\mathcal{A} \subset \TT^d \times \RR^n$, i.e. an open set homotopic to $\TT^d \times \mathcal{V}$, where $\mathcal{V} \subset \RR^n$ is an open set. Then, if $\mathcal{B} \subset \TT_{\CC}^d \times \CC^m$ is a complex neighborhood of $\mathcal{A}$, we denote by $(A_{\mathcal{B}})^m$ the Banach space of continuous functions $f: \bar{\mathcal{B}} \to \CC^m$, holomorphic on $\mathcal{B}$ such that $f(\mathcal{A}) \subseteq \RR^m$ (so $f$ is real analytic), with norm
\[
\|f\|_{\mathcal{B}} = \sup_{z \in \mathcal{B}} |f(z)|.
\]
Let $(A_{\mathcal{B}}^r)^m$ be the Banach space of continuous functions $f: \bar{\mathcal{B}} \to \CC^m$, holomorphic on $\mathcal{B}$ such that $f(\mathcal{A}) \in \RR^m$ and whose partial derivatives can be extended continuously to the boundary of $\mathcal{B}$, with norm
\[
\|f\|_{\mathcal{B},C^r} = \max_{\ell=0,\hdots,r} \sup_{z \in \mathcal{B}} |D^{\ell} f(z)|.
\]
In Theorem \ref{thm:kam} we will consider a $C^2$ family of real-analytic maps $F$ on the annulus (with bounded $C^2$-norms) homotopic to identity, which we will denote
as
\[
F \in \begin{pmatrix} x \\ 0 \end{pmatrix} + C^2(\mathcal{U},(A^2_{\mathcal{B}})^{d+n}),
\]
where $\mathcal{U}$ is an open set of parameters. So for each parameter value $a \in \mathcal{U}$ we get a map $F_a:\mathcal{B} \to \TT^d\times\CC^n$ given by $F_a := F(a)$. We will also consider these maps as functions $F:\mathcal{B}\times\mathcal{U} \to \TT^d\times\CC^n$ defined by $F(z,a) := F_a(z)$ for $z \in \mathcal{B}$, $a \in \mathcal{U}$, and write
\[
F(z,a) = \begin{pmatrix} x \\ 0 \end{pmatrix} + F_p(z,a), \quad z = \begin{pmatrix} x \\ y \end{pmatrix} \in \mathcal{B}, \: a \in \mathcal{U},
\]
where $F_p(\cdot,a) \in (A^2_{\mathcal{B}})^{d+n}$ for all $a \in \mathcal{U}$.

The main object of interest are parameterizations $K:\bar{\TT}_{\rho}^d \to \mathcal{B}$ of real-analytic tori $\mathcal{K} = K(\TT^d)$ (with bounded derivatives) homotopic to the zero-section, i.e.
\[
K(\theta) = \begin{pmatrix} \theta \\ 0 \end{pmatrix} + K_p(\theta) \quad \text{such that} \quad K_p \in (A^1_{\rho})^{d+n},
\]
and we will denote this as
\[
K \in \begin{pmatrix} \theta \\ 0 \end{pmatrix} + (A^1_{\rho})^{d+n}.
\]

\subsection{Quasi-Periodic Invariant Tori and Normal Hyperbolicity}

We say that the torus $\mathcal{K}$ parameterized by $K$ is a \textit{quasi-periodic $F_a$-invariant torus} with frequency $\omega \in \RR^d$ if $\omega$ is irrational, i.e.
\[
 \omega \cdot q - p \neq 0 \quad \text{for all } q \in \ZZ^d\setminus\{0\} \text{ and } p \in \ZZ,
\]
and the couple $(K,a)$ satisfies the equation
\[
F(K(\theta),a) = K(\theta+\omega), \quad \text{for all } \theta \in \TT^d.
\]
Here we can note that if $K$ is a parameterization of a quasi-periodic $F_a$-invariant torus, then so is $K_{\alpha}$ with $K_{\alpha}(\theta) = K(\theta+\alpha)$, $\theta \in \TT^d$, for any $\alpha \in \RR^d$. Because of this non-uniqueness it is appropriate to add an additional condition
\[
\langle K_{\alpha}^x(\theta) - \theta \rangle = 0,
\]
where the superindex $x$ represents the projection on the angle variable.

We say that the torus $\mathcal{K}$ parameterized by $K$ is a (real-analytic) \textit{quasi-periodic normally hyperbolic approximately $F_a$-invariant torus} with frequency $\omega \in \RR^d$ if there exists a normal bundle parameterized by $N^0 \in (A_{\rho}^0)^{(d+n) \times n}$ such that:
\begin{itemize}
\item The frame $P \in (A_{\rho}^0)^{(d+n) \times (d+n)}$, given by $P(\theta) = (DK(\theta),N^0(\theta))$, is invertible for all $\theta \in \TT^d$.
\item The vector bundle map $(H,R_{\omega}): \CC^n \times \bar{\TT}_{\rho}^d \to \CC^n \times \bar{\TT}_{\rho}^d$ given by
\[
(H,R_{\omega})(v,\theta) = (H(\theta)v, \theta+\omega),
\]
where
\[
H(\theta) = \begin{pmatrix} 0 & \Id_n \end{pmatrix} P^{-1}(\theta)D_z F(K(\theta),a) N^0(\theta)
\]
is uniformly hyperbolic, i.e. the induced transfer operator $\mathcal{H}: B(\bar{\TT}_{\rho}^d,\CC^n) \to B(\bar{\TT}_{\rho}^d,\CC^n)$ acting on the Banach space of bounded functions $B(\bar{\TT}_{\rho}^d,\CC^n)$, defined by
\[
\mathcal{H} \vartheta(\theta) = H(\theta-\omega)\vartheta(\theta-\omega)
\]
is hyperbolic, meaning that the operator $\mathcal{H} - \Id$ is invertible in $B(\bar{\TT}_{\rho}^d,\CC^n)$.
\end{itemize}

An equivalent condition to saying that $\mathcal{H}$ is hyperbolic is to say that for every $\eta \in B(\bar{\TT}_{\rho}^d,\CC^n)$ there is a unique solution $\xi \in B(\bar{\TT}_{\rho}^d,\CC^n)$ to the cohomological equation
\[
H(\theta)\xi(\theta) - \xi(\theta + \omega) = \eta(\theta), \quad \theta \in \TT^d,
\]
which we denote as $\xi = \mathcal{R}_H \eta$. This equivalence follows from the fact that $\mathcal{R}_H \eta(\theta) = (\mathcal{H} - \Id)^{-1} \eta(\theta - \omega)$ uniquely solves the above equation. So one practical way of showing that $\mathcal{H}$ is hyperbolic is to find a bound for
\[
\|\mathcal{R}_H\|_{\rho} := \|\mathcal{R}_H\|_{B(\bar{\TT}_{\rho}^d,\CC^n)}.
\]

We say that a (real analytic) quasi-periodic normally hyperbolic approximately $F_a$-invariant torus $\mathcal{K}$, with frequency $\omega \in \RR^d$ is \textit{non-degenerate} if the constant matrix $\langle B^L - T \mathcal{R}_H B^N \rangle$ is invertible, where
\[
B^L(\theta) = \begin{pmatrix} \Id_d & O \end{pmatrix} P(\theta+\omega)^{-1} D_a F(K(\theta),a),
\]
\[
B^N(\theta) = \begin{pmatrix} O & \Id_d \end{pmatrix} P(\theta+\omega)^{-1} D_a F(K(\theta),a),
\]
and
\[
T(\theta) = \begin{pmatrix} \Id_d & O \end{pmatrix} P(\theta+\omega)^{-1} D_z F(K(\theta),a) N^0(\theta).
\]

\subsection{Diophantine Vectors and R\"ussman Estimates} \label{sec:dio}

We formulated quasi-periodicity in terms of irrational rotation vectors $\omega$, but as is standard in KAM-theory we need to narrow down the vectors we consider in order to get useful results. We say a vector $\omega \in \RR^d$ is \textit{Diophantine of type $(\gamma,\tau)$} for $\gamma > 0$ and $\tau \geq d$ if
\[
| \omega \cdot q - p | \geq \frac{\gamma}{|q|^{\tau}}, \quad \text{for all } q \in \ZZ^d\setminus\{0\} \text{ and } p \in \ZZ.
\]
A vector $\omega$ is \textit{Diophantine of class $\tau$} if $\omega$ is Diophantine with constants $(\gamma,\tau)$ for some $\gamma>0$. A vector that is Diophantine of class $\tau$ for some $\tau \geq d$ is also sometimes referred to as simply a \textit{Diophantine} rotation vector.

When computing the constants for Theorem \ref{thm:kam} a value $c_R$ is required and it is determined by the Diophantine rotation vector $\omega$ used. The following lemma describes where this value comes from and we give one method for computing it.

\begin{lem}[R\"ussman Estimates]
Let $\omega$ be a Diophantine number of type $(\gamma,\tau)$, $\gamma > 0$, $\tau \geq d$. Then, there exists a constant $c_R > 0$ such that for any $\hat{\eta} \in A_{\rho}$ there exists a unique zero-average solution $\xi := \mathcal{R}(\hat{\eta})$ of the equation
\[
\xi(\theta) - \xi(\theta + \omega) = \hat{\eta}(\theta) - \langle \hat{\eta} \rangle,
\]
such that, for any $0 < \delta < \rho$, $\xi \in A_{\rho - \delta}$ and satisfies
\[
\|\xi\|_{\rho-\delta} \leq \frac{c_R}{\gamma\delta^{\tau}}\|\hat{\eta}\|_{\rho}.
\]
\end{lem}
Different values of $c_R$ satisfying this lemma have be derived and the one we will use is
\[
c_R = \frac{\sqrt{2^{d-3}\zeta(2,2^{\tau})\Gamma(2\tau+1)}}{(2\pi)^{\tau}},
\]
where $\zeta$ is Hurwitz zeta function and $\Gamma$ is the Gamma function. This formula for $c_R$ is classic and a derivation can be found in \cite[Lemma 4.20]{Haro_16}. Sharper bounds can be required for when choosing $\delta$ very small and \cite[Lemma 4.3]{Figueras_Haro_Luque_16} presents one such bound specifically suited for rigorous numerical evaluation.

\subsection{The KAM-like Theorem}

The following theorem states the precise conditions for when we can say that a quasi-periodic normally hyperbolic approximately $F_a$-invariant torus with a Diophantine rotation vector guarantees the existence of a quasi-periodic $F_{a^{\infty}}$-invariant torus with the same rotation vector for some parameter value $a^{\infty}$ near $a$. It also gives computable bounds on the distance between these tori and parameter values $a$ and $a^{\infty}$, as well as conditions for local uniqueness. This theorem and it's full proof is the topic of \cite{Canadell_Haro_2017}.

\begin{thm} \label{thm:kam}
Let $\mathcal{B} \subset \TT_{\CC}^d \times \CC^n$ be a complex neighborhood of the annulus $\mathcal{A} \subset \TT^d \times \RR^n$. Let
\[
F \in \begin{pmatrix} x \\ 0 \end{pmatrix} + C^2(\mathcal{U},(A^2_{\mathcal{B}})^{d+n})
\]
be a $C^2$-family of real analytic maps $F:\mathcal{B}\times\mathcal{U}\to\TT_{\CC}^d\times\CC^n$, homotopic to the identity, where $\mathcal{U} \subset \RR^d$ is the (open) set of parameters. Let $a\in\mathcal{U} \subset \RR^d$ be a parameter value.
Let $N^0\in(A_{\rho})^{(d+n) \times n}$ be a matrix valued map $N^0:\bar{\TT}_{\rho}^d \to \CC^{(d+n) \times n}$. Let
\[
K \in \begin{pmatrix} \theta \\ 0 \end{pmatrix} + (A^1_{\rho})^{d+n}
\]
be a homotopic to the zero section real-analytic paratmeterization $K: \bar{\TT}_{\rho}^d \to \mathcal{B}$ and assume that $\langle K^x(\theta) - \theta\rangle=0$. Notice that $\mathrm{dist}(K(\bar{\TT}_{\rho}^d),\partial\mathcal{B})>0$.
Let $\omega \in \RR^d$ be a frequency vector, and let $E \in (A_{\rho})^{d+n}$ be the error function $E:\bar{\TT}_{\rho}^d \to \CC^{(d+n)}$ defined by
\[
E(\theta) = F(K(\theta),a)-K(\theta+\omega).
\]
Assume there exists constants:
\begin{itemize}
\item[\textup{\textbf{H1}}] $c_{F,1,z},c_{F,1,a},c_{F,2}$ such that
\[
\|D_z F\|_{\mathcal{B}\times\mathcal{U}} < c_{F,1,z}, \: \|D_a F\|_{\mathcal{B}\times\mathcal{U}}<c_{F,1,a}, \: \|D^2 F\|_{\mathcal{B}\times\mathcal{U}} < c_{F,2},
\]
and $c_N$ such that $\|N^0\|_{\rho} < c_N$;
\item[\textup{\textbf{H2}}] $\sigma_L$ such that $\|DK\|_{\rho}<\sigma_L$;
\item[\textup{\textbf{H3}}] $\sigma_P$ such that $P=(DK \: N^0)$ is invertible with $\|P^{-1}\|_{\rho}<\sigma_P$;
\item[\textup{\textbf{H4}}] $\sigma_H$ such that the map $(H,R_{\omega})$ is hyperbolic with $\|\mathcal{R}_H\|_{\rho}<\sigma_H$.
\item[\textup{\textbf{H5}}] $\sigma_D$ such that the constant matrix $\langle B^L - T\mathcal{R}_H B^N\rangle$ is invertible with\\ ${\|\langle B^L - T\mathcal{R}_H B^N\rangle^{-1}\| < \sigma_D}$.
\item[\textup{\textbf{H6}}] $\gamma,\tau$ such that $| \omega \cdot q - p | \geq \gamma |q|^{-\tau},\,q\in \ZZ^d \setminus \{0\}$ and $p \in \ZZ$.
\end{itemize}
Then, for any $0<\delta<\frac{\rho}{2}$ and $0<\rho_{\infty}<\rho-2\delta$, there exits constants $\hat{C}_{*},\hat{C}_{**}$ and $\hat{C}_{***}$ (depending explicitly on the initial data, the initial analyticity strip $\rho$, the final analyticity strip $\rho_{\infty}$ and $\delta$) such that:
\begin{itemize}
\item[\textup{\textbf{T1}}] (Existence) If the following condition holds
\[
\frac{\hat{C}_{*} \|E\|_{\rho}}{\gamma^2 \rho^{2\tau}} < 1,
\]
then, there exists a couple $(K_{\infty},a_{\infty})$, with $K_{\infty}: \bar{\TT}_{\rho_{\infty}} \to \mathcal{B}$ such that
\[
K_{\infty} \in \left(\begin{array}{c}\theta\\0\end{array}\right) + (A^1_{\rho_{\infty}})^{d+n},
\]
and $a_{\infty} \in \mathcal{U}$, such that
\begin{align*}
F(K_{\infty}(\theta), a_{\infty})-K_{\infty}(\theta+\omega) &= 0, \\
\langle K_{\infty}^x(\theta) - \theta \rangle&= 0.
\end{align*}
That is, $\mathcal{K}_{\infty} = K_{\infty}(\hat{\TT}_{\rho_{\infty}})$ is QP-NHIT for $F_{a_{\infty}}$ with frequency $\omega$. Moreover, $K_{\infty}$ satisfies hypotheses \textup{H2-H5}.
\item[\textup{\textbf{T2}}] (Closeness) The torus $\mathcal{K}_{\infty}$ is close to the original approximation, in the sense that
\[
\|(K_{\infty} - K, a_{\infty} - a)\|_{\rho_{\infty}} \leq \frac{\hat{C}_{**}}{\gamma \rho^{\tau}} \|E\|_{\rho}
\]
\item[\textup{\textbf{T3}}] (Local Uniqueness) If the following conditions holds
\[
\frac{\hat{C}_{***}}{\gamma^2 \rho_{\infty}^{\tau} \rho^{\tau}} \|E\|_{\rho} < 1,
\]
then, if $(K_{\infty}', a_{\infty}')$ satisfies
\begin{align*}
F(K_{\infty}'(\theta), a_{\infty}') - K_{\infty}'(\theta+\omega) &= 0, \\
\langle {K_{\infty}'}^x(\theta) - \theta \rangle &= 0,
\end{align*}
and
\[
\|(K_{\infty}' - K, a_{\infty}' - a)\|_{\rho_{\infty}} <  \hat{C}_{**}\left( \frac{\gamma\rho_{\infty}^{\tau}}{\hat{C}_{***}} - \frac{1}{\gamma\rho^{\tau}} \|E\|_{\rho} \right),
\]
then, $(K_{\infty}', a_{\infty}') = (K_{\infty}, a_{\infty})$.
\end{itemize}

\end{thm}

A description for how to compute the constants $\hat{C}_{*}$, $\hat{C}_{**}$ and $\hat{C}_{***}$ can be found in Appendix \ref{sec:const}.

\section{Approximation and Estimation for Periodic Functions} \label{sec:AnEfPF}

In this section we present a selection of methods for working with periodic functions $f:\TT^d_{\rho} \to \CC$ given by their \textit{Fourier series}
\begin{equation*} 
f(\theta) = \sum_{k \in \ZZ^d} f_k e^{2\pi\im k \cdot \theta}, \quad \theta \in \TT^d,
\end{equation*}
where the \textit{Fourier coefficients} $f_k$ are given by
\begin{equation} \label{eq:FC}
f_k = \int_{\TT^d} f(\theta) e^{-2\pi\im k \cdot \theta} d\theta, \quad k \in \ZZ^d.
\end{equation}
We will follow the notation and definitions of \cite{Figueras_Haro_Luque_16} as we will later use some of the results from there. We refer to the norm $\|\cdot\|_{F,\rho}$ as the \textit{Fourier norm} and it is given by
\[
\| f \|_{F,\rho} = \sum_{k \in \ZZ^d} |f_k| e^{2\pi |k|_1 \rho},
\]
where $|k|_1 = \sum_{\ell=1}^d |k_{\ell}|$. We observe that $\|f\|_{\rho} \leq \|f\|_{F,\rho}$ for all $\rho>0$.

Now consider a sampling of points on a regular grid of size\\ ${N_F = ( N_{F,1}, \hdots, N_{F,d} ) \in \NN^d}$
\[
\theta_j := (\theta_{j_1}, \hdots, \theta_{j_d}) = \left( \frac{j_1}{N_{F,1}}, \hdots, \frac{j_d}{N_{F,d}} \right), \quad 0 \leq j < N_F,
\]
where  $j = (j_1, \hdots, j_d) \in \ZZ^d$, and $0 \leq j < N_F$ means that $0 \leq j_{\ell} < N_{F,\ell}$ for $\ell = 1,\hdots, d$. The total number of points is $N_D = N_{F,1} \cdots N_{F,d}$. Using the sampling of $f$ given by $f_j = f(\theta_j)$ we can then approximate the integral in \eqref{eq:FC} using the trapezoidal rule, thus obtaining the \textit{discrete Fourier transform} (DFT) of $\{f_j\}$
\[
\tilde{f}_k = (\DFT\{ f_j \})_k := \frac{1}{N_D} \sum_{0 \leq j \leq N_F} f_j e^{-2\pi\im k \cdot \theta_j}, \quad k \in \ZZ^d
\]
Here we note that $\tilde{f}_k$ is periodic over the components $k_1, \hdots, k_d$ with periods $N_{F,1}, \hdots, N_{F,d}$ respectively. The function $f$ can now be approximated by the \textit{discrete Fourier approximation}
\begin{equation} \label{eq:DFA}
\tilde{f}(\theta) = \sum_{k \in \mathcal{I}_{N_F}} \tilde{f}_k e^{2\pi\im k \cdot \theta}, \quad \theta \in \TT^d,
\end{equation}
where the sum is taken over the finite set
\[
\mathcal{I}_{N_F} = \left\{ k \in \ZZ^d : \frac{-N_{F,\ell}}{2} \leq k_{\ell} < \frac{N_{F,\ell}}{2}, \: \ell = 1, \hdots, d \right\}.
\]
Note that if we sample $\tilde{f}$ in \eqref{eq:DFA} at points $\{\theta_j\}$ we get the \textit{inverse discrete Fourier transform} (IDFT)
\[
\tilde{f}_j = (\IDFT\{f_k\})_j := \sum_{k \in \mathcal{I}_{N_F}} \tilde{f}_k e^{2\pi\im k \cdot \theta_j}, \quad 0 \leq j < N_F,
\]
and of course we know that $\IDFT = \DFT^{-1}$ so $\tilde{f}_j = f_j$ for $0 \leq j < N_F$.

\begin{rem}
We will make an effort to consistently use $k$ for indexing the frequency domain and $j$ for indexing the regular sampling of $\TT^d$, in order to avoid confusion over the ambiguity in notation like $f_k$ and $f_j$. This is also why we will use $i$ for indexing quasi-periodic orbits on $\TT^d$ in Section \ref{sec:WB}.
\end{rem}

In Section \ref{sec:WB} we present a method for approximating $f_k$ given values of $f$ on a quasi-periodic orbit of $\TT^d$. This is a non-rigorous method that will only be used for setting up the initial approximation. In Section \ref{sec:DFT} we look at approximating $f_k$ given values of $f$ on a regular grid of $\TT^d$. Here the error bounds are rigorous and can be used for computer assisted proofs. Lastly Section \ref{sec:ran} will deal with the problem of computing bounds for the range of $f$ given bounds for $f_k$.

\subsection{Approximation of Periodic Functions Using Weighted Birkhoff Sums} \label{sec:WB}

We say that $w:\RR \to [0,\infty)$ is a $C^r$ \textit{bump function}, $r \in \NN \cup \{\infty\}$, if $w \in C^r(\RR)$, such that $\supp w \subset [0,1]$ and $\int_{\RR}w(t)dt \neq 0$.

Let $s:\mathcal{M} \to E$ be a map from a manifold $\mathcal{M}$ to a vector space $E$ and let $T:\mathcal{M} \to \mathcal{M}$ be a map from $\mathcal{M}$ to itself. The given a $C^r$ bump function $w$ and starting point $z_0 \in \mathcal{M}$ we let the \textit{weighted Birkhoff average}, $\WB_N$, of $s$ starting in point $z_0$ be defined by
\[
(\WB_N s)(z_0) := \frac{1}{A_N}\sum_{i=0}^{N-1} w\left(\frac{i+1}{N+1}\right) s\left( T^i z_0 \right), \quad A_N := \sum_{i=0}^{N-1} w\left(\frac{i+1}{N+1}\right).
\]

\begin{thm} \label{thm:WB}
Let $m>1$ be an integer and let $w$ be a $C^r$ bump function for some $r \geq m$. For $M \in \NN$, let $\mathcal{M}$ be a $C^M$ manifold and $T: \mathcal{M} \to \mathcal{M}$ be a $C^M$ $d$-dimensional quasi-periodic map on $\mathcal{M}_0 \subseteq \mathcal{M}$, with invariant probability measure $\chi$ and a rotation vector of Diophantine class $\tau$. Let $s:\mathcal{M} \to E$ be $C^M$, where $E$ is a finite-dimensional, real vector space. Then there is a constant $C_m$ depending upon $w$, $s$, $m$, $M$ and $\tau$ but independent of $z_0 \in \mathcal{M}_0$ such that
\[
\left| (\WB_N s)(z_0) - \int_{\mathcal{M}_0} s \: d\chi \right| \leq C_m N^{-m}, \quad \text{for all } N>0,
\]
provided the ``smoothness'' M satisfies
\[
M > d + m(d+\beta).
\]
\end{thm}

For a proof of this theorem see \cite{Das_Yorke_18}. In particular we are interested when $\mathcal{M}$, $T$, $s$ and $w$ are all $C^{\infty}$, in which yields the following corollary.

\begin{cor} \label{cor:WB}
Let $\mathcal{M}$ be a $C^{\infty}$ manifold and $T:\mathcal{M} \to \mathcal{M}$ be a $d$-dimensional $C^{\infty}$ map which is quasi-periodic on $\mathcal{M}_0 \subset \mathcal{M}$, with invariant probability measure $\chi$ and a Diophantine rotation vector. Let $s:\mathcal{M} \to E$ be $C^{\infty}$, where $E$ is a finite-dimensional vector space. Assume $w$ is a $C^{\infty}$ bump function. Then for each $z_0 \in \mathcal{M}_0$, the weighted Birkhoff average $(\WB_N s)(z_0)$ converges super polynomially to $\int_{\mathcal{M}_0} f d\chi$, i.e. there exists constants $C_m$ such that
\[
\left| (\WB_N s)(z_0) - \int_{\mathcal{M}_0} s\: d\chi \right| \leq C_m N^{-m} \quad \text{for all } m>0 \text{ and } N>0.
\]
\end{cor}

What Theorem \ref{thm:WB} allows us to do is to take orbits on a quasi-periodic torus (or good approximations thereof) and approximate integrals on the torus with high accuracy. Many objects of interest can be written as such integrals, including rotation number and Lyapunov exponents. For our purpose we will be interested in the integrals for the Fourier coefficients of the torus. So if we let $\mathcal{M} = \mathcal{M}_0 = \TT^d$, $E = \CC$ and $T(\theta) = \theta + \omega$ where $\omega \in \RR^d$ is Diophantine, then for $s(\theta) = f(\theta)e^{-2\pi\im k \cdot \theta}$ we get that
\[
\int_{\mathcal{M}_0} s \: d\chi = \int_{\TT^d} f(\theta) e^{2\pi\im k \cdot \theta} d\theta = f_k,
\]
since the $T$-invariant measure on $\TT^d$ is given by $d\chi = d\theta$. Then we see that we can approximate $f_k$ using the sum
\[
(\WB_N s)(\theta_0) = \sum_{i=0}^{N-1} w_i f(\theta_0+i\omega) e^{-2\pi\im k \cdot (\theta_0+i\omega)},
\]
where $w_i := A_N^{-1} w\left(\frac{i+1}{N+1}\right)$.

It's worth noting here that one has to be careful with the almost resonances that occur when $k \cdot \omega$ is very close to an integer since then $e^{2\pi\im k \cdot \omega}$ is very close to $1$ and we have to take $N$ very large in order for $e^{2\pi\im k \cdot i\omega}$ to cover the unit circle in a satisfying way. The Diophantine properties of $\omega$ puts a bound on how bad this gets but care is still needed when choosing $N$.

\subsection{Approximation of Periodic Functions Using Discrete Fourier Transform} \label{sec:DFT}

While we know that $\tilde{f}(\theta_j) = f(\theta_j)$, $0 \leq j < N_F$, we will need uniform bounds of $|\tilde{f}(\theta) - f(\theta)|$, $\theta \in \TT_{\rho}^d$, as well as bounds on $|\tilde{f}_k - f_k|$ (although here we will only use the bound for $f_0 = \langle f \rangle$). Proposition \ref{prop:DFTFC} and \ref{prop:DFTFS} give explicit formulas for such bounds. In the case of $d=1$ these formulas can simplified and this is given in Corollary \ref{cor:DFTFC} and \ref{cor:DFTFS}.  All the propositions and corollaries with full proofs can be found in \cite{Figueras_Haro_Luque_16}.

\begin{prop} \label{prop:DFTFC}
Let $f:\TT^d_{\rho} \to \CC$ be an analytic and bounded function in the complex strip $\TT^d_{\rho}$ of size $\rho > 0$. Let $\tilde{f}$ be the discrete Fourier approximation of $f$ in the regular grid of size $N_F = (N_{F,1},\hdots,N_{F,d}) \in \NN^d$. Then, for $-\frac{N_F}{2} \leq k < \frac{N_F}{2}$:
\[
|\tilde{f}_k - f_k| \leq s_{N_F}^*(k,\rho) \|f\|_{\rho}
\]
where
\[
s_{N_F}^*(k,\rho) = \prod_{\ell = 1}^d \left( e^{-\pi\rho N_{F,\ell}} \frac{e^{2\pi \rho(|k_{\ell}| - N_{F,\ell}/2)} + e^{-2\pi\rho (|k_{\ell}| - N_{F,\ell}/2)}}{1 - e^{-2\pi\rho N_{F,\ell}}} \right) - e^{-2\pi\rho |k|_1}.
\]
\end{prop}

\begin{prop} \label{prop:DFTFS}
Let $f:\TT_{\hat{\rho}}^d \to \CC$ be an analytic and bounded function in the complex strip $\TT_{\hat{\rho}}^d$ of size $\hat{\rho} > 0$. Let $\tilde{f}$ be the discrete Fourier approximation of $f$ in the regular grid of size $N_F = (N_{F,1},\hdots,N_{F,d}) \in \NN^d$. Then,
\[
\|\tilde{f} - f\|_{\rho} \leq C_{N_F}(\rho,\hat{\rho}) \|f\|_{\hat{\rho}}
\]
for $0 \leq \rho < \hat{\rho}$, where $C_{N_F}(\rho,\hat{\rho}) = S_{N_F}^{*1}(\rho,\hat{\rho}) + S_{N_F}^{*2}(\rho,\hat{\rho}) + T_{N_F}(\rho,\hat{\rho})$ is given by
\[
S_{N_F}^{*1}(\rho,\hat{\rho}) = \prod_{\ell = 1}^d \frac{1}{1 - e^{-2\pi \hat{\rho} N_{F,\ell}}} \sum_{\substack{\sigma \in \{-1,1\}^d \\ \sigma \neq (1,\hdots,1) }} \prod_{\ell=1}^d e^{(\sigma_{\ell}-1)\pi\hat{\rho} N_{F,\ell}} \nu_{\ell}(\sigma_{\ell} \hat{\rho} - \rho),
\]
\[
S_{N_F}^{*2}(\rho,\hat{\rho}) = \prod_{\ell = 1}^d \frac{1}{1 - e^{-2\pi \hat{\rho} N_{F,\ell}}} \left( 1 - \prod_{\ell=1}^d \left( 1 - e^{-2\pi\hat{\rho} N_{F,\ell}} \right) \right) \prod_{\ell=1}^d \nu_{\ell}(\hat{\rho} - \rho)
\]
and
\[
T_{N_F}^{*2}(\rho,\hat{\rho}) = \left( \frac{e^{2\pi(\hat{\rho} - \rho)} + 1}{e^{2\pi(\hat{\rho} - \rho)} - 1} \right)^d \left( 1 - \prod_{\ell=1}^d\left( 1 - \mu_{\ell}(\hat{\rho} - \rho) e^{-\pi(\hat{\rho} - \rho) N_{F,\ell}} \right) \right)
\]
with
\[
\nu_{\ell}(\delta) = \frac{e^{2\pi\delta} + 1}{e^{2\pi\delta} - 1} \left( 1 - \mu_{\ell}(\delta)e^{-\pi\delta N_{F,\ell}} \right) \quad \text{and} \quad \mu_{\ell}(\delta) =
\begin{cases}
1 & \text{if } N_{F,\ell} \text{ is even,} \\
\frac{2e^{\pi\delta}}{e^{2\pi\delta} + 1} & \text{if } N_{F,\ell} \text{ is odd.}
\end{cases}
\]
\end{prop}

\begin{cor} \label{cor:DFTFC}
Let $f: \TT_{\hat{\rho}} \to \CC$ be an analytical and bounded function. Let $\{\tilde{f}_k\}$ be the discrete approximations of the Fourier coefficients $\{f_k\}$ of $f$ in the regular grid of size $N_F \in \NN$. Then for $k = -\left[\frac{N_F}{2}\right],\hdots,\left[\frac{N_F-1}{2}\right]$,
\[
|\tilde{f}_k-f_k| \leq s_{N_F}^*(k,\rho) \|f\|_{\rho}
\]
where
\[
s_{N_F}^*(k,\rho) = \frac{e^{-2\pi \rho N_F}}{1 - e^{-2\pi \rho N_F}}\left(e^{2\pi \rho k} + e^{-2\pi \rho k}\right).
\]
\end{cor}

\begin{cor} \label{cor:DFTFS}
Let $f:\TT_{\hat{\rho}}\to\CC$ be an analytic and bounded function in the complex strip $\TT_{\hat{\rho}}$ of size $\hat{\rho}>0$. Let $\tilde{f}$ be the discrete Fourier approximation of $f$ in the regular grid of size $N_F \in \NN$. Then, for $0\leq\rho<\hat{\rho}$, we have
\[
\|\tilde{f}  - f\|_{\rho} \leq C_{N_F}(\rho,\hat{\rho})\|f\|_{\hat{\rho}},
\]
where $C_{N_F}(\rho,\hat{\rho}) = S_{N_F}^{*1}(\rho,\hat{\rho})+S_{N_F}^{*2}(\rho,\hat{\rho})+T_{N_F}(\rho,\hat{\rho})$, with
\[
S_{N_F}^{*1} = \frac{e^{-2\pi\hat{\rho}N_F}}{1-e^{-2\pi\hat{\rho}N_F}} \frac{e^{-2\pi(\hat{\rho}+\rho)}+1}{e^{-2\pi(\hat{\rho}+\rho)}-1} \left( 1 - e^{\pi(\hat{\rho}+\rho)N_F} \right),
\]
\[
S_{N_F}^{*2} = \frac{e^{-2\pi\hat{\rho}N_F}}{1-e^{-2\pi\hat{\rho}N_F}} \frac{e^{2\pi(\hat{\rho}-\rho)}+1}{e^{2\pi(\hat{\rho}-\rho)}-1} \left( 1 - e^{-\pi(\hat{\rho}-\rho)N_F} \right),
\]
and
\[
T_{N_F}(\rho,\hat{\rho}) = \frac{e^{2\pi(\hat{\rho}-\rho)}+1}{e^{2\pi(\hat{\rho}-\rho)}-1} e^{-\pi(\hat{\rho}-\rho)N_F}.
\]
\end{cor}

\subsection{Enclosing the Range of a Function Using Discrete Fourier Transform} \label{sec:ran}

Enclosing the range of a function is a well studied problem in validated numerics and several different methods are available depending on how the function is represented. For a function defined in terms of a finite Fourier series there is a convenient way utilizing the DFT to efficiently calculate an enclosure of the range.

For a function $f:\TT^d \to \CC$ given by the finite Fourier series $f(\theta) = \sum_{k \in \mathcal{I}_{N_F}} f_k e^{2\pi\im k \cdot \theta}$ assume we have enclosures $E_{f_k} \supseteq f_k$ of the Fourier coefficients and introduce the covering $\{\Theta_j\}_{0 \leq j < N_F}$ of $\TT^d$ defined by
\[
\Theta_j = \Theta_{j_1} \times \hdots \times \Theta_{j_n} = (\theta_{j_1} + U_{N_{F,1}}) \times \cdots \times (\theta_{j_d} + U_{N_{F,d}}) = \theta_j + U_{N_F},
\]
where $U_{N_F} = U_{N_{F,1}} \times \hdots \times U_{N_{F,d}}$ is given by
\[
U_{N_{F,\ell}} = \left[0,\frac{1}{N_{F,\ell}}\right], \quad 1 \leq \ell \leq d.
\]
Now consider the image of $\Theta_j$ under $f$
\[
f(\Theta_j) = \sum_{k \in \mathcal{I}_{N_F}} f_k e^{2\pi\im k \cdot \Theta_j} = \sum_{k \in \mathcal{I}_{N_F}} f_k e^{2\pi\im k \cdot U_{N_F}} e^{2\pi\im k \cdot \theta_j},
\]
so using the enclosures $E_{f_k}$ we see that
\[
f(\Theta_j) \subseteq \left( \IDFT \{ E_{f_k} e^{2\pi\im k \cdot U_{N_F}} \}_{k \in \mathcal{I}_{N_F}} \right)_j
\]
and enclosing the union of the right hand side over all $j$ gives us an enclosure  of $f(\TT^d)$. The advantage of expressing this in terms of the IDFT is that we can use a fast Fourier transform algorithm in our calculations.

Later on we are going to need an enclosure for $|f(\TT_{\rho}^d)|$ and we can extend the above enclosure to get one. So assume we have $|f(\TT^d)| \subseteq [a,b]$, $0 \leq a \leq b$, and note that
\[
f(\theta + \im y) - f(\theta) = \sum_{k \in \mathcal{I}_{N_F}} f_k e^{2\pi\im k \cdot \theta}(e^{-2\pi k \cdot y} - 1).
\]
Since this is a holomorphic functions it's maximum modulus on $\TT_{\rho}^d$ is attained at the boundary, hence
\[
\sup_{\theta + \im y \in \TT_{\rho}^d} |f(\theta + \im y) - f(\theta)| \leq \sum_{k \in \mathcal{I}_{N_F}} (e^{2\pi |k|_1 \rho} - 1) |f_k| = \|f\|_{F,\rho} - \|f\|_{F,0}.
\]
Therefore we see that
\[
|f(\TT_{\rho}^d)| \subseteq \left[\max(a-\|f\|_{F,\rho}+\|f\|_{F,0},0),b+\|f\|_{F,\rho}-\|f\|_{F,0}\right].
\]
We will also need to bound the imaginary part of $f(\TT_{\rho}^d)$ where we assume that $f(\TT^d) \subset \RR$. Then we can split the (finite) Fourier series in two parts as
\[
f(\theta + \im y) = \sum_{k \in \mathcal{I}_{N_F}} f_k \cosh(2\pi k \cdot y) e^{2\pi\im k \cdot \theta} + \im \sum_{k \in \mathcal{I}_{N_F}} \im f_k \sinh(2\pi k \cdot y) e^{2\pi\im k \cdot \theta},
\]
and note that
\[
f_{-k} \cosh(2\pi (-k) \cdot y) = \overline{f_k \cosh(2\pi k \cdot y)}
\]
and
\[
\im f_{-k} \sinh(2\pi (-k) \cdot y) = \overline{\im f_k \sinh(2\pi k \cdot y)}
\]
since $f_{-k} = \overline{f_k}$. This means that both sums are real valued and therefore make up the real and imaginary part of $f(\theta + \im y)$. Since f is holomorphic the imaginary part is harmonic and achieves its maximum on the boundary and due to symmetry it is achieved at both boundaries. Hence $\IM f(\TT_{\rho}^d) \leq \IM f(\TT^d + \im \rho)$ which we can enclose using our earlier method.

\section{A Validation Algorithm For the Dissipative Standard Map} \label{sec:validation}

In this section we will go through the details of an implementation of Theorem \ref{thm:kam} for the dissipative standard map given by
\begin{equation}\label{eq:map}
F(z,\mu) = \left( \begin{array}{c}
x + \varphi_{\mu}(x,y)\\
\varphi_{\mu}(x,y)
\end{array} \right)
= \left(\begin{array}{c} x \\ 0 \end{array} \right) + F_p(z,\mu),
\end{equation}
where $z = (x,y)^T \in \TT\times\RR$ and $ \varphi_{\mu}(x,y) = \lambda y + \mu + \frac{\varepsilon}{2\pi} \sin(2\pi x)$. The parameters $\lambda,\varepsilon \in (0,1)$ are fixed and $\mu$ will be used as the tuning parameter $a$. Note here that $d = n = 1$ for this map.

To do this we must estimate constants satisfying the following
\begin{align*}
\|D_z F\|_{\mathcal{B}\times\mathcal{U}}  < c_{F,1,z}, && \|D_{\mu} F\|_{\mathcal{B}\times\mathcal{U}}< c_{F,1,a}, && \|D^2 F\|_{\mathcal{B}\times\mathcal{U}} < c_{F,2}, \\
\|N^0\|_{\rho} < c_N, && \|DK\|_{\rho} \leq c_L < \sigma_L, && \|P^{-1}\| \leq c_P < \sigma_P \\
\|\mathcal{R}_H\|_{\rho} \leq c_H < \sigma_H, && |\langle B^L - T\mathcal{R}_HB^N\rangle| \leq c_D < \sigma_D, && \|E\|_{\rho} \leq c_E
\end{align*}
where $\mathcal{B} = \TT\times\RR$ and $\mathcal{U} = \RR$. Note that for the various pairs of $c_*$ and $\sigma_*$, we want $c_*$ to be as sharp as we can get with a reasonable computation load and then $\sigma_*$ is then chosen close to $c_*$ but without $(\sigma_* - c_*)^{-1}$ becoming to large.

In addition we will also calculate $\hat{c}_N$ and $\hat{c}_L$ such that
\[
\|N^0\|_{\hat{\rho}} < \hat{c}_N \quad \text{and} \quad \|DK\|_{\hat{\rho}} < \hat{c}_L.
\]
These will be used to calculate $c_H$ and $c_D$.

\subsection{Calculating \texorpdfstring{$c_{F,1,z}$}{cF1z}, \texorpdfstring{$c_{F,1,a}$}{cF1a} and \texorpdfstring{$c_{F,2}$}{cF2}} \label{sec:cF}

We see that
\[
D_z F(z;\mu) =  \left(\begin{array}{cc} 1 + \varepsilon\cos(2\pi x) & \lambda \\ \varepsilon\cos(2\pi x) & \lambda \end{array} \right), \hspace{4mm} D_{\mu} F(z;\mu) = \left(\begin{array}{c} 1 \\ 1 \end{array} \right),
\]
and all second order derivatives are zero except for $D_x^2 F(z;\mu) = -2\pi \varepsilon \sin(2\pi x)$, so we can choose
\[
c_{F,1,z} = 1+\lambda+e^{2\pi\rho}, \quad c_{F,1,\mu} = 1 + \epsilon_M, \quad c_{F,2} = 2\pi e^{2\pi\rho},
\]
where $\epsilon_M$ is the smallest number such that $1<1 + \epsilon_M$ in the precision used.

\subsection{Finding an Approximation \texorpdfstring{$K_p$}{Kp}} \label{sec:approxK}

Iterative application of the map $F_{\mu}$ on a point in the vicinity of the attracting circle yields an orbit we can use as an approximation for the circle. However this gives us an approximation of $K$, not $K_p$ which is needed to approximate the Fourier coefficients. To deal with this we will make use of the following quasi-periodic skew product system $Q_{\omega}: \RR^2 \times \TT \to \RR^2 \times \TT$ given by
\[
Q_{\omega} \begin{pmatrix} x \\ y \\ \theta \end{pmatrix} = \begin{pmatrix} x - \omega + \varphi(x + \theta, y) \\ \varphi(x + \theta, y) \\ \theta + \omega \end{pmatrix}.
\]
Then if we introduce the map $\Psi:\RR^2 \times \TT \to \TT \times \RR$ given by
\[
\Psi\begin{pmatrix} x \\ y \\ \theta \end{pmatrix} = \begin{pmatrix} \theta \\ 0 \end{pmatrix} + \begin{pmatrix} x \\ y \end{pmatrix}
\]
we see that $\Psi$ is a semiconjugacy between $F_a$ and $Q_{\omega}$ since
\[
F_a \circ \Psi  \begin{pmatrix} x \\ y \\ \theta \end{pmatrix} =  \begin{pmatrix} \theta + \omega \\ 0 \end{pmatrix} + \begin{pmatrix} x - \omega + \varphi(x + \theta, y) \\ \varphi(x + \theta, y) \end{pmatrix} = \Psi \circ Q_{\omega} \begin{pmatrix} x \\ y \\ \theta \end{pmatrix}.
\]
The point here is that if we have a parameterization $K_Q$ of a tori such that
\[
K_Q(\theta) = \begin{pmatrix} K_p^x(\theta) \\ K_p^y(\theta) \\ \theta \end{pmatrix}, \quad \theta \in \TT,
\]
then
\[
K(\theta) := \Psi \circ K_Q(\theta) = \begin{pmatrix} \theta \\ 0 \end{pmatrix} + \begin{pmatrix} K_p^x(\theta) \\ K_p^y(\theta) \end{pmatrix} = \begin{pmatrix} \theta \\ 0 \end{pmatrix} + K_p(\theta), \quad \theta \in \TT,
\]
satisfies
\[
|F \circ K (\theta) - K(\theta + \omega)| = |Q_{\omega} \circ K_Q(\theta) - K_Q(\theta+\omega)|, \quad \theta \in \TT.
\]

So in order to approximate the Fourier coefficients we take an orbit under $Q_{\omega}$ of size $N_O$ (starting very close to the attractor) ,
\[
\begin{pmatrix} x_{i} \\ y_{i} \\ \theta_{i} \end{pmatrix} = Q_{\omega} \begin{pmatrix} x_{i-1} \\ y_{i-1} \\ \theta_{i-1} \end{pmatrix}, \quad i = 1,\hdots,N_O-1.
\]
and disregard the angle component when storing the values
\[
\quad z_i = \begin{pmatrix} x_i \\ y_i \end{pmatrix}, \quad i = 0,\hdots,N_O-1.
\]
We then follow the outline in Section \ref{sec:WB} for how to approximate Fourier coefficients using weighted Birkhoff sums. So choose $K_p(\theta) = \sum_{|k|<N_A} K_k e^{2\pi\im k\theta}$, where $N_A > 0$ is the number of frequencies we want to use for the approximation, and the Fourier coefficients are given by
\[
K_k = \sum_{i=0}^{N_O-1} w_i z_i e^{-2\pi\im ki \omega}, \quad |k| < N_A,
\]
where we use the $C^{\infty}$ bump function
\[
w(t) = \begin{cases}
\exp\left( \frac{1}{t(t-1)} \right) & \text{if } t \in (0,1), \\
0 & \text{if } t \notin (0,1).
\end{cases}
\]

\subsection{Calculating \texorpdfstring{$c_N$}{cN}, \texorpdfstring{$c_L$}{cL}, \texorpdfstring{$c_P$}{cP}, \texorpdfstring{$\hat{c}_N$}{cN} and \texorpdfstring{$\hat{c}_L$}{cL}} \label{sec:cNLP}

Let $P$ be the adapted frame given by $P(\theta) = (DK(\theta), N^0(\theta))$, where
\[
N^0(\theta) = \left(\begin{array}{cc} 0 & -1 \\ 1 & 0 \end{array} \right) DK(\theta) \left(DK(\theta)^T DK(\theta)\right)^{-1}.
\]
Here we note that
\[
P(\theta) = \left( \begin{array}{cc} DK_x(\theta) & N_x^0(\theta) \\ DK_y(\theta) & N_y^0(\theta) \end{array} \right) \quad \text{and} \quad P(\theta)^{-1} = \left( \begin{array}{cc} N_y^0(\theta) & -N_x^0(\theta) \\ -DK_y(\theta) & DK_x(\theta) \end{array} \right)
\]
since $\det{P(\theta)} = 1$ for all $\theta \in \TT_{\CC}$.

Now to get the constants we simply estimate a bound for the range $|DK(\TT_{\rho})| \subset [\underline{r},\overline{r}]$ using the methods in Section \ref{sec:ran}. Now  we can take $c_L = \overline{r}$ and $c_N = \underline{r}^{-1}$ since
\[
|N^0(\theta)| < |N^0(\theta)|_2 = |DK(\theta)|_2^{-1} < |DK(\theta)|^{-1}
\]
where $|z|_2 = \sqrt{|x|^2+|y|^2}$, $z=(x,y)^T \in \CC^2$. $\hat{c}_N$ and $\hat{c}_L$ are calculated in the same way by enclosing the range of $|DK(\TT_{\hat{\rho}})|$. For $c_P$ we simply note that
\[
|P(\theta)^{-1}| = \max(|N^0(\theta)|_1,|DK(\theta)|_1) \leq 2 \max(|N^0(\theta)|,|DK(\theta)|).
\]
so we can take $c_P = 2\max(c_L,c_N)$.

\subsection{Calculating \texorpdfstring{$c_H$}{cH}} \label{sec:cH}

We begin by noting that we can write $h$ as
\[
h(\theta) = (\varepsilon\cos(2\pi K_x(\theta))(DK_x(\theta+\omega) - DK_y(\theta+\omega)) - DK_y(\theta+\omega) )N^0_x(\theta)
\]
and see that we can directly evaluate $h$ at grid points $\{\theta_j\}$ using interval arithmetic. Then we can use Corollary \ref{cor:DFTFS} to bound $\|h\|_{\rho}$, however we first need a bound on $\|h\|_{\hat{\rho}}$. Using the above formula we see that
\[
\|h\|_{\hat{\rho}} \leq (1 + 2\varepsilon + 2\lambda)\hat{c}_N\hat{c}_L
\]
so
\[
\|h\|_{\rho} \leq \|\tilde{h}\|_{F,\rho} + C_{N_F}(\rho,\hat{\rho})((1 + 2\varepsilon + 2\lambda)\hat{c}_N\hat{c}_L).
\]

\begin{lem}
If $H(\theta) = \lambda  + h(\theta)$, $\theta \in \TT_{\rho}$, and $\|h\|_{\rho} < 1 - \lambda$, then the operator $\mathcal{R}_H$ can be bounded by
\[
\| \mathcal{R}_H\|_{\rho} \leq \frac{1}{1-\lambda-\|h\|_{\rho}}.
\]
\end{lem}
\begin{proof}
For analytical functions $\xi,\eta:\TT_{\rho}\to\CC$ solving the cohomological equation
\[
H(\theta)\xi(\theta) - \xi(\theta+\omega) = \eta(\theta), \quad \theta \in \TT_{\rho},
\]
we have that $\mathcal{R}_H\eta = \xi$. Now consider the cohomological equation given by
\[
\lambda\xi^*(\theta) - \xi^*(\theta+\omega) = \eta(\theta), \quad \theta \in \TT_{\rho},
\]
and denote the solution as $\mathcal{R}_{\lambda} \eta = \xi^*$. Then if we write the first equation as
\[
\lambda \xi(\theta) - \xi(\theta+\omega) = \eta(\theta) - h(\theta)\xi(\theta)
\]
it should be clear that
\[
\mathcal{R}_H\eta = \mathcal{R}_{\lambda}\left( \eta - h \mathcal{R}_H \eta \right)
\]
which gives us
\[
\|\mathcal{R}_H\|_{\rho} \leq \|\mathcal{R}_{\lambda}\|_{\rho} \left( 1 - \|h\|_{\rho} \|\mathcal{R}_H\|_{\rho}\right).
\]
Now using the fact that $\|\mathcal{R}_{\lambda}\|_{\rho} = (1-\lambda)^{-1}$ we can solve this inequality for the desired bound.
\end{proof}
Combining this lemma with the bound for $\|h\|_{\rho}$ gives us the constant
\[
c_H = \frac{1}{1 - \lambda - \|\tilde{h}\|_{F,\rho} - C_{N_F}(\rho,\hat{\rho})((1 + 2\varepsilon + 2\lambda)\hat{c}_N\hat{c}_L)},
\]
if $1 - \lambda > \|\tilde{h}\|_{F,\rho} + C_{N_F}(\rho,\hat{\rho})((1 + 2\varepsilon + 2\lambda)\hat{c}_N\hat{c}_L)$

\subsection{Calculating \texorpdfstring{$c_D$}{cD}} \label{sec:cD}

First we see that
\[
\langle B^L - T\mathcal{R}_H B^N \rangle = \langle B^L - T\mathcal{R}_{\lambda} B^N \rangle - \langle T e_{\mathcal{R}}B^N \rangle
\]
where $e_{\mathcal{R}} = \mathcal{R}_H - \mathcal{R}_{\lambda}$, so
\[
|\langle B^L - T\mathcal{R}_H B^N \rangle|^{-1} \leq (|\langle B^L - T\mathcal{R}_{\lambda} B^N \rangle| - \|T\|_{\rho} \|e_{\mathcal{R}}\|_{\rho}\|B^N\|_{\rho})^{-1}, \quad \rho > 0.
\]
\begin{rem}
Since $\langle \cdot \rangle$ is the same for all $\rho$ we don't need to use the same $\rho$ here as in the previous estimates, but we will do so anyway since we can reuse bounds already calculated.
\end{rem}
We note that $T$, $B^L$ and $B^N$ can be given explicitly by
\[
T(\theta) = (\varepsilon \cos(2\pi K_x(\theta))(N_y^0(\theta+\omega) - N_x^0(\theta+\omega)) + N_y^0(\theta+\omega))N_x^0(\theta)
\]
\[
+ \lambda(N_y^0(\theta+\omega) - N_x^0(\theta+\omega)) N_y^0(\theta),
\]
\[
B^L(\theta) = N_y^0(\theta+\omega) - N_x^0(\theta+\omega) \quad \text{and} \quad B^N(\theta) = DK_x(\theta+\omega) - DK_y(\theta+\omega).
\]
We see here that $B^N$ has Fourier coefficients $B_k^N = 2\pi\im((K_k)_x-(K_k)_y)e^{2\pi\im k\omega}$ and we calculate $\mathcal{R}_{\lambda} B^N$ through
\[
\mathcal{R}_{\lambda} B^N(\theta) = \sum_{|k| < N_A} \frac{B_k^N}{\lambda - e^{2\pi\im k\omega}} e^{2\pi\im k\theta}.
\]
It is now clear that $T$, $B^L$ and $\mathcal{R}_{\lambda} B^N$ can be evaluated on the grid $\{\theta_j\}$ since we have the Fourier series for $DK$ and $\mathcal{R}_{\lambda} B^N$ and $N^0$ is given as a formula of $DK$.
So let $\Phi = B^L - T\mathcal{R}_{\lambda} B^N$ and let $\tilde{\Phi}$ be the Fourier approximation obtained from $\{\Phi(\theta_j)\}$. From Corollary \ref{cor:DFTFC} we get that
\[
|\langle B^L - T\mathcal{R}_{\lambda} B^N \rangle| \geq |\langle\tilde{\Phi}\rangle| - s_{N_F}^*(0,\rho)(\|B^L\|_{\rho} + \|T\|_{\rho} \|\mathcal{R}_{\lambda} B^N\|_{\rho}).
\]
since $\Phi_k = \langle \Phi \rangle$ for $k=0$, and we note that $\langle\tilde{\Phi}\rangle = N_F^{-1}\sum_j \Phi_j$. To estimate the norms we see that
\[
\|T\|_{\rho} \leq \|\tilde{T}\|_{F,\rho} + (1 + 2\varepsilon + 2\lambda) \hat{c}_N^2C_{N_F}(\rho,\hat{\rho}),
\]
\[
\|B^L\|_{\rho} \leq \|\tilde{B}^L\|_{F,\rho} + 2\hat{c}_NC_{N_F}(\rho,\hat{\rho}),
\]
\[
\|B^N\|_{\rho} \leq \|B^N\|_{F,\rho} \quad \text{and} \quad \|\mathcal{R}_{\lambda} B^N\|_{\rho} \leq \|\mathcal{R}_{\lambda} B^N\|_{F,\rho}
\]
\begin{lem}
If $H = \lambda + h$ and $\|h\|_{\rho} < 1 - \lambda$, then
\[
\|\mathcal{R}_H - \mathcal{R}_{\lambda}\|_{\rho} \leq \frac{\|h\|_{\rho}}{(1-\lambda)(1-\lambda-\|h\|_{\rho})}.
\]
\end{lem}
\begin{proof}
If $\mathcal{R}_H\eta = \xi$ and $\mathcal{R}_{\lambda}\eta = \xi^*$ then by subtracting the corresponding cohomological equations we see that
\[
\xi(\theta+\omega) - \xi^*(\theta+\omega) = \lambda(\xi(\theta)-\xi^*(\theta)) + h(\theta)\xi(\theta)
\]
which gives us
\[
\|\mathcal{R}_H - \mathcal{R}_{\lambda}\|_{\rho} \leq \lambda\|\mathcal{R}_H - \mathcal{R}_{\lambda}\|_{\rho} + \|h\|_{\rho} \|\mathcal{R}_H\|_{\rho}.
\]
Using that $\|\mathcal{R}_H\|_{\rho} \leq (1 - \lambda - \|h\|_{\rho})^{-1}$ we see that the above inequality can be solved for $\|\mathcal{R}_H - \mathcal{R}_{\lambda}\|_{\rho}$.
\end{proof}
Combining all of this we get $c_D = (r_2^* - r_3^*)^{-1}$ where
\begin{align*}
r_1^* &= \|\tilde{T}\|_{F,\rho} + (1 + 2\varepsilon + 2\lambda) \hat{c}_N^2C_{N_F}(\rho,\hat{\rho}), \\
r_2^* &=  |\langle\tilde{\Phi}\rangle| - s_{N_F}^*(0,\rho)(\|\tilde{B}^L\|_{F,\rho} + 2\hat{c}_NC_{N_F}(\rho,\hat{\rho}) + r_1^* \|\mathcal{R}_{\lambda} B^N\|_{F,\rho}), \\
c_h &= \|\tilde{h}\|_{F,\rho} + C_{N_F}(\rho,\hat{\rho})((1 + 2\varepsilon + 2\lambda)\hat{c}_N\hat{c}_L),\\
r_3^* &= r_1^* \frac{c_h}{(1-\lambda)(1-\lambda-c_h)} \|B^N\|_{F,\rho}.
\end{align*}

\subsection{Calculating \texorpdfstring{$c_E$}{cE}} \label{sec:cE}

Let $G(\theta) = F(K(\theta),\mu)$ and $K_{\omega}(\theta) = K(\theta+\omega)$ so that $E(\theta) = G(\theta) - K_{\omega}(\theta)$, $\theta \in \TT_{\hat{\rho}}$. Here we should note that $G(\theta)$ and $K_{\omega}(\theta)$ are values on a torus so subtraction between them is technically not well defined and we should instead understand $E$ as given by $E(\theta) = G_p(\theta) - K_{p,\omega}(\theta)$ where
\[
G(\theta) = \begin{pmatrix} \theta \\ 0 \end{pmatrix} + G_p(\theta) \quad \text{and} \quad K_{\omega}(\theta) = \begin{pmatrix} \theta \\ 0 \end{pmatrix} + K_{p,\omega}(\theta).
\]
Note that we can calculate $G_p(\theta)$ from $K_p(\theta)$ using the map $Q_{\omega}$ from Section \ref{sec:approxK}. Then we have
\[
\|E\|_{\rho} \leq \|\tilde{G}_p - K_{p,\omega}\|_{\rho} + \|G_p - \tilde{G}_p\|_{\rho} \leq  \|\tilde{G}_p - K_{p,\omega}\|_{F,\rho} + C_{N_F}(\rho,\hat{\rho}) \|G_p\|_{\hat{\rho}}.
\]
To get a bound on $\|G_p\|_{\hat{\rho}}$ we see that
\[
\|G_p\|_{\hat{\rho}} \leq \sup_{\theta\in\TT_{\hat{\rho}}} |K_p^x(\theta)| + \lambda|K_p^y(\theta)| + \mu + \frac{\varepsilon}{2\pi} |\sin(2\pi (\theta + K_p^x(\theta)))|
\]
and if we calculate $\bar{\rho}$ such that $K_p^x(\TT_{\hat{\rho}}) \subset \TT_{\bar{\rho}}$, by estimating the range of $K_p$ as explained in Section \ref{sec:ran}, we get
\[
\|G_p\|_{\hat{\rho}} \leq (1 + \lambda)\hat{c}_L + \mu + \frac{\varepsilon}{2\pi} e^{2\pi( \hat{\rho} + \bar{\rho})}.
\]
Hence we take
\[
c_E =  \|\tilde{G}_p - K_{p,\omega}\|_{F,\rho} + C_{N_F}(\rho,\hat{\rho})\left((1 + \lambda)\hat{c}_L + \mu + \frac{\varepsilon}{2\pi} e^{2\pi(\hat{\rho} + \bar{\rho})}\right).
\]

\subsection{The Algorithm}

Here we sum up the results of this Section and describe an algorithm that can check the condition for existence in Theorem \ref{thm:kam} applied to the dissipative standard map in \eqref{eq:map}. The algorithm takes as input the parameters $\lambda$, $\mu$ and $\varepsilon$, and an enclosure of the rotation vector $\omega$ along with it's Diophantine constants $\gamma$ and $\tau$.

The constants $c_{F,1z}$, $c_{F,1,\mu}$ and $c_{F,2}$ are given in Section \ref{sec:cF} and can be calculated with only the parameters. The constant $c_R$ is calculated using $\gamma$ and $\tau$ as described in Section \ref{sec:dio}. In order to calculated the rest of the constants we must first get the Fourier coefficients $\{K_k\}_{|k|<N_A}$ of an approximate invariant torus as described in Section \ref{sec:approxK}.

Now it can be very nonobvious how to choose $N_F$ along with $\delta$, $\rho$, $\hat{\rho}$ and $\rho_{\infty}$ in order to verify the condition for existence. Choosing $N_F$ too large can result in unpractical computation times and the other parameters usually have to be balanced so that neither $\hat{C}^*$ becomes too large nor $\rho^{2\tau}$ too small. One possible strategy which we will use is to simply test your way through a list of values. So we let $N_F$ be taken from an ordered set $\mathcal{N}_F$ and $(\delta,\rho,\hat{\rho},\rho_{\infty})$ from $\mathcal{P}_S$. Here we note that we don't need to recalculate everything when changing values for $(\delta,\rho,\hat{\rho},\rho_{\infty})$, so we need to determine what can be calculated without these values.

To calculate $c_N$, $\hat{c}_N$, $c_L$, $\sigma_L$, $\hat{c}_L$ and $c_P$ we see from Section \ref{sec:cNLP} we only need $\{K_k\} = \{K_k\}_{-\frac{N_F}{2}\leq k < \frac{N_F}{2}}$ in addition to $\rho$ and $\hat{\rho}$ so no preparation can be done here. For $c_H$ and $\sigma_H$ we see in Section \ref{sec:cH} that these are calculated from $\{\tilde{h}_k\}$ and $\rho$ so here it makes sense to compute $\{\tilde{h}_k\}$ beforehand and save between changing $\rho$. For $c_D$ and $\sigma_D$ we saw Section \ref{sec:cD} that we need $\{\tilde{h}_k\}$, $\{\tilde{T}_k\}$, $\{\tilde{B}_k^L\}$, $\{B_k^N\}$, $\{(\mathcal{R}_{\lambda}B^N)_k\}$ and $\{\tilde{\Phi}_k\}$ in addition to $\rho$ so all those Fourier coefficients can be saved between changing $\rho$. With these constants we can calculate $\hat{C}^*$ as described in Appendix \ref{sec:const}. To calculate $c_E$ we need the Fourier coefficients for $\tilde{G}_p - K_{p,\omega}$, which we will denote as $\{\tilde{G}_k-K_{k,\omega}\}$, in addition to $\rho$ and $\hat{\rho}$. In Algorithm 1 we write the whole algorithm in pseudocode.

\begin{algorithm}[H]
\SetKw{KwReturn}{return}
\SetKw{KwTrue}{True}
\SetKw{KwFalse}{False}
\KwData{Parameters $\lambda$, $\mu$, $\varepsilon$ and $\omega$ along with $\gamma$ and $\tau$}
\KwResult{\KwTrue if the program can prove the existence of an invariant torus and \KwFalse if it cannot }
Calculate constants $c_{F,1z}$, $c_{F,1,\mu}$, $c_{F,2}$ and $c_R$\;
Approximate Fourier coefficients $\{K_k\}_{|k|<N_A}$\;
\For{$N_F \in \mathcal{N}_{F}$}{
	Calculate enclosures for $\{\tilde{h}_k\}$, $\{\tilde{T}_k\}$, $\{\tilde{B}_k^L\}$, $\{B_k^N\}$, $\{(\mathcal{R}_{\lambda}B^N)_k\}$, $\{\tilde{\Phi}_k\}$ and $\{\tilde{G}_k-K_{\omega,k}\}$ for $-\frac{N_F}{2} \leq k < \frac{N_F}{2}$\;
	\For{$(\delta,\rho,\hat{\rho},\rho_{\infty}) \in \mathcal{P}_S$}{
		Calculate $c_N$, $\hat{c}_N$, $c_L$, $\sigma_L$, $\hat{c}_L$, $c_P$, $\sigma_P$, $c_H$, $\sigma_H$, $c_D$, $\sigma_D$, $\hat{C}_*$, $c_E$\;
		\If{$\hat{C}_{*} c_E < \gamma^2\rho^{2\tau}$}{
			\KwReturn \KwTrue\;
		}
	}
}
\KwReturn \KwFalse\;
\caption{Proving the existence of an invariant tori}
\end{algorithm}

\begin{rem}
If the algorithm returns \textbf{True} we can of course also calculate $\hat{C}_{**}$ and $\hat{C}_{***}$ and to get an estimate on how close the invariant torus is and if there is local uniqueness, if those things are of interest.
\end{rem}

\section{Numerical Results} \label{sec:results}

In order to test the practicality of our approach we implement Algorithm 1 on a computer. The implementation is done in C++ with multiple precision floating point numbers and intervals provided by the MPFR and MPFI libraries respectively. We use a custom FFT based on the implementation found in \cite{PresTeukVettFlan92}.

Care must taken when choosing parameters, both for the map and for the algorithm. For rotation number we choose the golden mean $\omega = \frac{\sqrt{5}-1}{2}$ because it has very strong Diophantine properties and from \cite{Figueras_Haro_Luque_16} we get that we can take $\tau = 1.26$ and $\gamma = 0.381966011250104$. We fix $\lambda = 0.4$ (this was the value considered in \cite{Calleja_Figueras_12}), then for any value of $\varepsilon$ we can estimate $\mu$ using the methods described in \cite{Linroth_19}. We take orbit size $N_O = 65296$ and $N_A = 240$ number of Fourier modes. We choose $N_{F}$ for an increasing list of powers of $2$ and $\rho$ is chosen from a decreasing list of floating point numbers, we fix $\hat{\rho} \approx 10^{-2}$ and let $\rho_{\infty} = \delta = \rho/4$. The results of the algorithm can be seen in Table \ref{tab:1}, where we recall that Theorem \ref{thm:kam} asserts existence of an invariant circle if $\frac{\hat{C}_{*}\|E\|_{\rho}}{\gamma^2 \rho^{2\tau}} < 1$ and $\frac{\hat{C}_{**}\|E\|_{\rho}}{\gamma \rho^{\tau}}$ is a bound on the distance between the approximation and the invariant circle. Note here that we are not looking to minimize any of the constants but simply to prove existence of an invariant circle.

\renewcommand{\arraystretch}{1.2}
\begin{table}[h!]
\begin{tabular}{ p{1cm} p{2cm} p{2cm} p{2cm} p{2cm} } 
  $\varepsilon$ & $\rho$ & $N_F$ & $\frac{\hat{C}_{*}\|E\|_{\rho}}{\gamma^2 \rho^{2\tau}}$ & $\frac{\hat{C}_{**}\|E\|_{\rho}}{\gamma \rho^{\tau}}$ \\[5pt]
  \hline
  0.1 & $1.95 \cdot 10^{-3}$ & $2048$ & $2.75 \cdot 10^{-3}$ & $9.74 \cdot 10^{-13}$ \\
  0.2 & $1.95 \cdot 10^{-3}$ & $2048$ & $6.82 \cdot 10^{-3}$ & $1.67 \cdot 10^{-12}$ \\
  0.3 & $1.95 \cdot 10^{-3}$ & $2048$ & $2.00 \cdot 10^{-2}$ & $3.14 \cdot 10^{-12}$ \\
  0.4 & $1.95 \cdot 10^{-3}$ & $2048$ & $7.29 \cdot 10^{-2}$ & $6.70 \cdot 10^{-12}$ \\
  0.5 & $1.95 \cdot 10^{-3}$ & $2048$ & $4.17 \cdot 10^{-1}$ & $1.83 \cdot 10^{-11}$ \\
  0.6 & $9.77 \cdot 10^{-4}$ & $2048$ & $3.74 \cdot 10^{-2}$ & $1.83 \cdot 10^{-13}$ \\
  0.7 & $3.9 \cdot 10^{-3}$ & $4096$ & $1.72 \cdot 10^{-7}$ & $3.21 \cdot 10^{-15}$ \\
  \hline
\end{tabular}
\caption{Results of Algorithm 1 proving existence of invariant circle with golden mean rotation for different values of $\varepsilon$.}
\label{tab:1}
\end{table}

\bibliography{bibliography}{}
\bibliographystyle{alpha}

\newpage
\appendix

\section{Computing the Constants for the KAM-like Theorem} \label{sec:const}

If we have upper bounds 
\[
    c_L \geq \|DK\|_{\rho}, \hspace{3mm} c_P \geq \|P^{-1}\|_{\rho}, \hspace{3mm} c_H \geq \|\mathcal{R}_H\|_{\rho},
\]
\[
    c_D \geq |\langle B^L - T \mathcal{R}_{H}(B^N) \rangle^{-1}|, \hspace{3mm} c_B \leq \mathrm{dist}(K(\TT_{\rho}^d),\partial \mathcal{B})
\]
such that $c_L < \sigma_L$, $c_P < \sigma_P$, $c_H < \sigma_H$ and $c_D < \sigma_D$, as well as lower bounds $c_U \leq \mathrm{dist}(a,\partial \mathcal{U})$ we can compute the constants $\hat{C}_{*}$, $\hat{C}_{**}$ and $\hat{C}_{***}$ in Theorem \ref{thm:kam} with the following formulas
\begin{align*}
C_1 &= 1 + \sigma_H \sigma_P c_N c_{F,1,z}, \\
C_2 &= 1 + \sigma_P \sigma_D c_{F,1,a} C_1, \\
C_3 &= C_2((\sigma_L + 1) c_R \sigma_P C_1 + c_N \sigma_H \sigma_P \gamma \delta^{\tau}), \\
\hat{C}_2 &= \sigma_L C_3 + c_N \sigma_H \sigma_P C_2 \gamma \delta^{\tau}, \\
\hat{C}_3 &= \sigma_D \sigma_P C_1, \\
\hat{C}_{2,3} &= \max(\hat{C}_2,\hat{C}_3 \gamma \delta^{\tau}), \\
\hat{C}_4 &= 2 \sigma_P^2 d \hat{C}_2, \\
C_4 &= c_N(\sigma_P c_{F,2} \hat{C}_{2,3} \delta + c_{F,1,z}\hat{C}_4), \\
\hat{C}_5 &= 2 \sigma_H^2 C_4, \\
C_5 &= \hat{C}_4 c_{F,1,a} + \sigma_P c_{F,2} \hat{C}_{2,3} \delta, \\
C_6 &= C_5 + \sigma_H \sigma_P c_{F,1,a} C_4 + \sigma_P c_{F,1,z} c_N \sigma_H C_5 + \sigma_P c_{F,1,z} c_N \sigma_P c_{F,1,a} \hat{C}_5, \\
\hat{C}_6 &= 2 \sigma_D^2 C_6, \\
\hat{C}_7 &= d c_R C_3 \gamma \delta^{\tau-1} + \frac{1}{2} c_{F,2} \hat{C}_{2,3}, \\
a_1 &= \frac{\rho-\rho_{\infty}}{\rho-2\delta-\rho_{\infty}}, \quad a_3 = \frac{\rho}{\delta}, \\
\hat{C}_8 &= \max\left\{ \frac{d\hat{C}_2}{\sigma_L - c_L} \frac{1}{1-a_1^{1-\tau}}, \frac{\hat{C}_4}{\sigma_P - c_P} \frac{1}{1-a_1^{1-\tau}}, \right. \\
& \frac{\hat{C}_5}{\sigma_H - c_H} \frac{1}{1-a_1^{1-\tau}}, \frac{\hat{C}_6}{\sigma_D - c_D} \frac{1}{1-a_1^{1-\tau}}, \\
& \left.  \frac{\hat{C}_2 \delta}{c_B} \frac{1}{1-a_1^{-\tau}}, \frac{\hat{C}_3 \delta^{\tau+1} \gamma}{c_U} \frac{1}{1-a_1^{-2\tau}}  \right\}, \\
\hat{C}_* &= \max\left\{ \hat{C}_7 (a_1 a_3)^{2\tau}, \hat{C}_8a_3^{\tau+1} \gamma \rho^{\tau-1} \right\}, \\
\hat{C}_{**} &= \max\left\{ \frac{\hat{C}_2 a_3^{\tau}}{1-a_1^{-\tau}}, \frac{\hat{C}_3 \gamma \rho^{\tau}}{1-a_1^{-2\tau}} \right\}, \\
\tilde{C} &= \hat{C}_{2,3} \frac{1}{2} c_{F,2}, \\
\hat{C}_{***} &= 4^{\tau} \tilde{C} \hat{C}_{**}.
\end{align*}

\end{document}